\documentclass[10pt]{amsart}
\usepackage{amssymb,mathtools}
\usepackage{amsmath}

\usepackage{blindtext,stmaryrd}
\usepackage[utf8]{inputenc}
\usepackage{xcolor}
\usepackage{hyperref}
\usepackage{enumerate}
\newtheorem{thm}{Theorem}[section]
\theoremstyle{definition}

\newtheorem{conj}[thm]{Conjecture}

\newtheorem{cor}[thm]{Corollary}

\theoremstyle{definition}
\newtheorem{defn}[thm]{Definition}
\theoremstyle{definition}

\theoremstyle{remark}

\numberwithin{equation}{section}

\newcommand{\Z}{{\mathbb Z}}

\newcommand{\N}{{\mathbb N}}

\begin{document}

\title[Sarnak's conjecture for rank-one subshifts]{Sarnak's conjecture for rank-one subshifts}

\author[M. Etedadialiabadi]{Mahmood Etedadialiabadi}

\address{Department of Mathematics, University of North Texas, 1155 Union Circle \#311430, Denton, TX 76203, USA}

\email{mahmood.etedadialiabadi@unt.edu}

\author[S. Gao]{Su Gao}

\address{Department of Mathematics, University of North Texas, 1155 Union Circle \#311430, Denton, TX 76203, USA}\email{sgao@unt.edu}

\thanks{The second author's research was partially supported by the U.S. NSF grant DMS-1800323. The authors would like to thank Cesar Silva for discussions on the topic of Sarnak's conjecture.}



\begin{abstract}
Using techniques developed in \cite{KLR}, we verify Sarnak's conjecture for two classes of rank-one subshifts with unbounded cutting parameters. The first class of rank-one subshifts we consider are called {\em almost complete congruency classes} ({\em accc}), the definition of which is motivated by the main result of \cite{GS}, which implies that, when a rank-one subshift carries a unique non-atomic invariant probability measure, it is accc if it is measure-theoretically isomorphic to an odometer. The second class we consider consists of Katok's map and its generalizations.
\end{abstract}

\maketitle


\section{Introduction}
The M\"{o}bius function, $\mu:\mathbb{N}\rightarrow \{-1,0,1\}$, is defined such that: $\mu(n)=0$ if $n$ is divisible by $p^2$ for some prime number $p$; and $\mu(n)=(-1)^k$ if $n=p_1p_2\cdots p_k$ where $p_1,p_2,\dots,p_k$ are distinct prime numbers. The M\"{o}bius function is one of the most important functions in Number Theory, and in particular the study of the M\"{o}bius function is highly consequential in Analytical Number Theory. For instance, the fact that the respective numbers of $1$s and $-1$s as values of the M\"{o}bius function are almost the same is equivalent to the prime number theorem.
\begin{thm}[Landau; see \cite{Apo} \S4.9]\label{PNT}
The statement that $\sum_{n\leq N} \mu(n)=o(N)$ is equivalent to the prime number theorem (PNT).
\end{thm}

 Furthermore, the Riemann hypothesis can be restated in terms of the rate of cancellation in $\sum_{n\leq N} \mu(n)$.
\begin{thm}[Littlewood, 1912; see \cite{T}]
The Riemann hypothesis is equivalent to the statement that for every $\epsilon>0$ we have $\sum_{n\leq N} \mu(n)=o(N^{1/2+\epsilon})$.
\end{thm}

In this paper we concentrate on the study of the random behavior of the M\"{o}bius function and not necessarily the speed of the cancellation. One of the strongest conjectures on the random nature of the sequence $\{\mu(n)\}_{n\in \mathbb{N}}$ is due to Chowla.
\begin{conj}[Chowla]
Let $0\leq i_1, i_2, \cdots, i_k\leq 2$ be a sequence of integers with at least one taking value 1. Then
\[
\sum_{n\leq N} \mu(n+1)^{i_1}\mu(n+2)^{i_2}\cdots \mu(n+k)^{i_k}=o(N).
\]
\end{conj}

Chowla's conjecture seems out of reach for the moment and a weaker notion (see \cite{AKPLR}, Theorem 4.10; \cite{Ta}) of pseudorandomness for the M\"{o}bius function, Sarnak's conjecture, is the main focus of the present work. In an attempt to formalize the random behavior of M\"{o}bius function using tools from Dynamical Systems, Sarnak suggested the following conjecture.

\begin{conj}[Sarnak]
Let $X$ be a compact metric space and $T:X\rightarrow X$ be a homeomorphism. If the topological dynamical system $(X,T)$ is of entropy zero, then we have
\[
\sum_{n\leq N} f(T^nx)\mu(n)=o(N)
\]
for every continuous function $f:X\rightarrow \mathbb{R}$ and $x\in X$.
\end{conj}

Following \cite{KLR}, we say that $(X,T)$ is {\em M\"{o}bius disjoint} if
\[
\sum_{n\leq N} f(T^nx)\mu(n)=o(N)
\]
for every continuous function $f:X\rightarrow \mathbb{R}$ and $x\in X$. Furthermore, we say a continuous function $f:X\rightarrow \mathbb{R}$ satisfies {\it Sarnak's property} if
\[
\sum_{n\leq N} f(T^nx)\mu(n)=o(N)
\]
for every $x\in X$.

Sarnak's conjecture has been studied extensively in recent years (see, for example, \cite{ALR}, \cite{Bourgain}, \cite{FKPL}, \cite{FM}, \cite{HLSY}, \cite{KLR}, and \cite{KLR2}). In particular, the study of the conjecture for symbolic systems corresponding to the class of rank-one transformations is of interest.

Given sequences of positive integers $r_n>1$ for $n\in\mathbb{N}$ and nonnegative integers $s_{n,i}$ for $n\in\mathbb{N}$ and $0<i\leq r_n$, define a {\em generating sequence} $v_n$ of finite words recursively by setting $v_0=0$ and
\begin{equation}\label{GG} v_{n+1}=v_n1^{s_{n,1}}v_n1^{s_{n,2}}\cdots v_n1^{s_{n, r_n}} \end{equation}
for $n\in \mathbb{N}$. An {\em infinite rank-one word} $V\in 2^{\mathbb{N}}$ is then defined as $V=\lim_{n\to\infty} v_n$ and the {\em rank-one subshift} $(X_V, T)$ is given by
$$ X_V=\{ x\in 2^{\mathbb{Z}}\,:\, \mbox{every finite subword of $x$ is a subword of $V$}\} $$
and $T(x)(a)=x(a+1)$ for all $x\in X_V$ and $a\in \mathbb{Z}$. The sequences $(r_n)_{n\in\mathbb{N}}$ and $(s_{n,i})_{n\in\mathbb{N}, 0< i\leq r_n}$ are known as, respectively, the {\em cutting parameter} and the {\em spacer parameter} of the rank-one subshift. A rank-one subshift $(X_V, T)$ is {\em nontrivial} if $X_V$ is infinite, or equivalently, $V$ is aperiodic. In this paper we only consider nontrivial rank-one subshifts. Note that a rank-one subshift is always of topological entropy zero. $(X_V, T)$ is {\em bounded} if there is $M>0$ such that $r_n<M$ and $s_{n,i}<M$ for all $n\in\mathbb{N}$ and $0<i\leq r_n$.

Bourgain \cite{Bourgain} proved Sarnak's conjecture for bounded rank-one subshifts for the special case that $s_{n,r_n}=0$ for all $n\in\mathbb{N}$. This was extended to all bounded rank-one subshifts by El Abdalaoui--Lema\'{n}czyk--de la Rue \cite{ALR}.

\begin{thm}[Bourgain \cite{Bourgain}; El Abdalaoui--Lema\'{n}czyk--de la Rue \cite{ALR}]
Let $(X,T)$ be a bounded rank-one subshift. Then $(X,T)$ is M\"{o}bius disjoint.
\end{thm}

In this paper we consider two classes of rank-one subshifts with unbounded cutting parameters. The consideration of the first class is motivated by the main result of \cite{GS}. In that context the authors assumed that the generating sequence satisfies the condition
\begin{equation}\label{CC} \displaystyle\sum_{n=0}^\infty \frac{|v_{n+1}|-r_n|v_n|}{|v_{n+1}|}<+\infty,
\end{equation}
which guarantees that $(X_V, T)$ admits a unique non-atomic invariant probability measure $\mu$. The main result of \cite{GS} is a characterization of when the measure-preserving transformation $(X_V, \mu, T)$ is measure-theoretically isomorphic to an odometer. To state this characterization we need to make the following definition. For $n\geq m$, apply (\ref{GG}) inductively to write $v_n$ uniquely in the form
$$ v_n=v_m 1^{b_1} v_m 1^{b_2}\cdots v_m1^{b_t} $$
and let $I_{m,n}$ be the set of indices for the starting positions of the copies of $v_m$ (starting at index $0$ for the starting position of the first copy). Note that $I_{0,n}$ is the set of positions of all $0$s in $v_m$.

\begin{thm}[Foreman--Gao--Hill--Silva--Weiss \cite{GS}]\label{Getal} The rank-one measure-preserving transformation $(X_V, \mu, T)$ is measure-theoretically isomorphic to an odometer if and only if for all $l \in \N$ and all $\epsilon>0$, there is some $k \in \N$ such that for all $\eta >0$ there exists an $N \in \N$ such that for all $n > m \geq N$,
\begin{enumerate}
\item[\rm (a)]  there is some $j \in \Z / k\Z$ such that $$\frac{|\{i \in I_{m,n}: [i]_k \neq j\}|}{|I_{m,n}|} < \eta, \mbox{ and}$$
\item[\rm (b)]  there is some $D \subseteq \Z / k\Z$ such that $$\frac{|\{i \leq |v_m|: [i]_k \in D\} \triangle I_{l,m} |}{|I_{l,m}|} < \epsilon,$$
\end{enumerate}
where $[i]_k$ denotes the congruency class of $i\!\mod k$.
\end{thm}

Note that if $(X_V, \mu, T)$ is isomorphic to an odometer and $(X_V, T)$ is bounded, then $V$ is periodic and $(X_V, T)$ is trivial. Motivated by Clause (b) with $l=0$, we introduce the following notion.

\begin{defn}
Let $M\subseteq \mathbb{Z}$ be nonempty and $A\subseteq \mathbb{N}$ be finite. We say $A$ is {\it a building block of $M$} if $0\in A$ and there exists a non-decreasing sequence $\{a_{i,A}\}_{i\in\mathbb{Z}}$ of integers such that ${M=\bigcup_{i\in\mathbb{Z}} (A+a_{i,A})}$ and for every $i\in \mathbb{Z}$ we have $a_{i+1,A}-a_{i,A}>\max(A)$ or $a_{i+1,A}=a_{i,A}$.
\end{defn}

\begin{defn}
We say $M\subseteq \mathbb{Z}$ is an {\it almost complete congruency class} ({\em accc}) if $M=\emptyset$ or for every $\epsilon>0$ there exist $k\in\mathbb{N}$ with the following property which we denote as $P(M,\epsilon, k)$:
\begin{quote}
 for every  $N\in\mathbb{N}$ there exist a building block $A$ of $M$ and $D_A\subseteq\mathbb{Z}/k\mathbb{Z}$ such that $\max(A)\geq N$ and
$$\displaystyle\frac{\lvert \{ 0\leq n\leq \max(A) \ : \ [n]_k\in D_A\}\triangle A \rvert}{\max(A)} <\epsilon.$$
\end{quote}
\end{defn}

Thus Theorem~\ref{Getal} implies that if $(X_V, \mu, T)$ is isomorphic to an odometer then the set of positions for $0$s in $V$ is an accc, and in this case it also follows that for every $x\in X_V$, the set of positions for $0$s in $x$ is an accc. Motivated by this observation, we call a rank-one subshift $(X,T)$ an {\it accc rank-one subshift} if $$M_x=\{n\in \mathbb{Z} \ : \ x(n)=0\}$$ is an accc for every $x\in X$.

Our first main result of the paper is the following.

\begin{thm}
Let $(X,T)$ be an accc rank-one subshift. Then $(X, T)$ is M\"{o}bius disjoint.
\end{thm}

This theorem will be proved in Section 2. In Sections 3 and 4 we will consider another class of rank-one subshifts which are generalizations of Katok's map studied in \cite{ALR}. Katok's map is a rank-one subshift where for all $n\in\mathbb{N}$, $r_n$ is even and
$$ s_{n,i}=\left\{\begin{array}{ll} 0, & \mbox{ for $0< i\leq r_n/2$,} \\ 1, & \mbox{ for $r_n/2<i\leq r_n$.} \end{array}\right. $$
In \cite{ALR} Sarnak's conjecture for Katok's map was verified under the condition $$\displaystyle\lim_{n\to\infty} \frac{r_n}{|v_n|}=+\infty.$$ Here we prove Sarnak's conjecture for a class of generalized Katok's maps under a weaker condition.

The key technique used in all of our proofs is an estimate of the M\"{o}bius function on short intervals along arithmetic progressions developed by Kanigowski--Lema\'{n}czyk--Radziwi\l\l \ \cite{KLR}.
\begin{thm}[Kanigowski--Lema\'nczyk--Radziwi\l\l \ \cite{KLR}]\label{KLRadziwill}
For each $\epsilon\in (0,\frac{1}{100})$ there exists $L_0$ such that for each $L\geq L_0$ and $q\geq 1$ with
\begin{equation}\label{longap}
\sum_{\substack{p\vert q\\ p \text{ prime}}}\frac{1}{p}\leq \ (1-\epsilon)\!\!\sum_{\substack{p\leq L\\ p \text{ prime}}}\frac{1}{p}
\end{equation}
we can find $N_0=N_0(q,L)$ such that for all $N\geq N_0$, we have
\begin{equation}
\sum_{j=0}^{N/Lq}\sum_{a=0}^{q-1}\Bigr\lvert \sum_{\substack{m\in[z+jLq,z+(j+1)Lq)\\m\equiv a \bmod{q}}}\mu(m) \Bigr\rvert\leq \epsilon N
\end{equation}
for some $0\leq z<Lq$.

\end{thm}

\section{Accc Rank-One Subshifts}
\begin{defn}
Let $M\subseteq \mathbb{Z}$. We say $M$ is {\em orthogonal to the M\"{o}bius function} if
\[
\lim_{N\rightarrow \infty} \frac{1}{N}\sum_{\substack{n\in M\\ 1\leq n\leq N}} \mu(n)=0.
\]
\end{defn}
Note that by Theorem \ref{PNT}, if $\mathbb{N}\subseteq M\subseteq \mathbb{Z}$ then $M$ is orthogonal to the M\"{o}bius function. Trivially the empty set is orthogonal to the M\"{o}bius function.


\begin{thm}\label{multiplication}
Let $M\subseteq \mathbb{Z}$ be an accc and $n_1,n_2,\dots,n_l$ be integers. Then
\[
M'=\{n\in \mathbb{Z} \ : \ n+n_1,n+n_2,\dots,n+n_l\in M\}
\]
is orthogonal to the M\"{o}bius function.
\end{thm}

\begin{proof}
Fix $0<\epsilon<\frac{1}{100}$. Assume $M$ is nonempty. Since $M$ is an accc, there exists $k\in \mathbb{N}$ such that property $P(M,\epsilon,k)$ holds. Applying Theorem \ref{KLRadziwill} with $q=k$, we obtain $L_0$ and $L\geq L_0$ satisfying (\ref{longap}), and there exists $N_0=N_0(k,L)>0$ such that for every $N\geq N_0$ we have
\[
\sum_{j=0}^{N/Lk}\sum_{a=0}^{k-1}\Bigr\lvert \sum_{\substack{m\in[z+jLk,z+(j+1)Lk)\\m\equiv a \bmod{k}}}\mu(m) \Bigr\rvert\leq \epsilon N
\]
for some $0\leq z< Lk$.

From property $P(M,\epsilon,k)$ we obtain a building block $A$ of $M$ with
$$\max(A)\geq 2N_0Lk/\epsilon, \lvert 2n_1/\epsilon\rvert,\dots,\lvert 2n_l/\epsilon\rvert,$$ $D_A\subseteq \mathbb{Z}/k\mathbb{Z}$, and a non-decreasing sequence $\{a_{i,A}\}_{i\in\mathbb{Z}}$ of integers such that for every $i\in \mathbb{Z}$ we have $a_{i+1,A}-a_{i,A}> \max(A)$ or $a_{i+1,A}=a_{i,A}$, ${M=\bigcup_{i\in\mathbb{Z}} (A+a_{i,A})}$, and
\[
\frac{\lvert \{ 0\leq n\leq \max(A)  :  [n]_k\in D_A\}\triangle A \rvert}{\max(A)} <\epsilon.
\]

Let $$D_A'=\bigcap_{\lambda=1}^l (-n_\lambda+D_A)\subseteq \mathbb{Z}/k\mathbb{Z}.$$ Fix an arbitrary $i\in \mathbb{Z}$. Let $B=\{0\leq n\leq \max(A) :  [n]_k\in D_A'\}$ and
$B_i=\{0\leq n\leq \max(A): n+a_{i,A}\in M'\}$.
We claim that
\[\displaystyle\frac{\lvert B\triangle B_i\rvert}{\max(A)} \leq(l+1)\epsilon.
\]
To see this, let $\delta=\max(\lvert n_1\rvert,\dots,\lvert n_l\rvert)$, $C=\{\delta\leq n\leq \max(A)-\delta :  [n]_k\in D_A'\}$, and
$C_i=\{\delta\leq n\leq \max(A)-\delta:n+a_{i,A}\in M'\}$. Note that
 $$ |(C\triangle C_i)\triangle (B\triangle B_i)|\leq 2\delta\leq \epsilon\max(A). $$ It is therefore sufficient to verify that $|C\triangle C_i|\leq l\epsilon\max(A)$. Observe that if $\delta\leq n\leq \max(A)-\delta$ and $1\leq \lambda\leq l$, then $n+a_{i,A}+n_\lambda\in M$ iff $n+n_\lambda\in A$. Fix any $n$ with $\delta\leq n\leq \max(A)-\delta$. If $n+a_{i,A}\in M'$ and $[n]_k\notin D_A'$, then there exists $1\leq \lambda\leq l$ such that $n+n_\lambda\in A$ and $[n+n_\lambda]_k\notin D_A$. Similarly, if $n+a_{i,A}\notin M'$ and $[n]_k\in D_A'$, then there exists $1\leq \lambda\leq l$ such that $n+n_\lambda\notin A$ and $[n+n_\lambda]_k\in D_A$. In either case, we have
 \[
n+n_\lambda\in \{ 0\leq m\leq \max(A)  : [m]_k\in D_A\}\triangle A.
\]
Now $|C\triangle C_i|\leq l\epsilon\max(A)$ follows from the fact that we have $l$-many different possibilities for $\lambda$ and
\[
\lvert \{ 0\leq m\leq \max(A)  :  [m]_k\in D_A\}\triangle A\rvert\leq\epsilon \max(A).
\]
This proves the claim.

Let $M''=\bigcup_{i\in\mathbb{Z}} (B+a_{i,A})$. We next claim that for every $N\geq 2\max(A)/\epsilon$, we have
\[
\frac{\lvert \{n\in M' : 1\leq n\leq N\} \triangle \{n\in M'' :  1\leq n\leq N\} \rvert}{N} \leq(l+1)\epsilon+\epsilon.
\]
To see this, let $s=\min\{i:1\leq a_{i,A}\leq N\}$ and $r=\max\{i: 1\leq a_{i,A}\leq N\}$. Then $(r-s)\max(A)\leq N$. Note that
$$ M'\cap [1,N]=\bigcup_{i=s}^{r-1}(B_i+a_{i,A})\cup D_0 $$
and
$$ M''\cap [1,N]=\bigcup_{i=s}^{r-1} (B+a_{i,A})\cup D_1 $$
for some $D_0, D_1\subseteq [1,a_{s,A})\cup [a_{r,A}, N]$. Thus by the preceding claim we have
$$\begin{array}{rcl} |(M'\cap [1,N])\triangle (M''\cap [1,N])|& \leq &2\max(A)+\displaystyle\sum_{i=s}^{r-1}|B\triangle B_i| \\
&\leq &2\max(A)+(r-s)(l+1)\epsilon\max(A) \\ \\
&\leq &\epsilon N+(l+1)\epsilon N=(l+2)\epsilon N. \end{array}
$$
It now follows that
\[
\Bigr\lvert\frac{1}{N}\sum_{\substack{n\in M'\\ 1\leq n\leq N}} \mu(n)\Bigr\rvert\leq (l+2)\epsilon+\Bigr\lvert\frac{1}{N}\sum_{\substack{n\in M''\\ 1\leq n\leq N}} \mu(n)\Bigr\rvert.
\]

We conclude the proof by showing that, for $N>3\max(A)/\epsilon$,
\[
\Bigr\lvert\sum_{\substack{n\in M''\\ 1\leq n\leq N}} \mu(i)\Bigr\rvert\leq \epsilon N+\epsilon N+\sum_{j=0}^{N/Lk}\sum_{a=0}^{k-1}\Bigr\lvert \sum_{\substack{m\in[z+jLk,z+(j+1)Lk)\\m\equiv a \bmod{k}}}\mu(m) \Bigr\rvert\leq 3 \epsilon N.
\]
To see this, for each $s\leq i\leq r$, let $t_i=\max\{j\in\mathbb{Z}: z+jLk\leq a_{i,A}\}$ and $u_i=\max\{j\in\mathbb{Z}: z+jLk\leq a_{i,A}+\max(A)\}$. Then in the above inequality the first error term of $\epsilon N$ allows us to consider, instead of $M''\cap[1,N]$, the set $M''\cap [z+t_sLk, z+t_rLk)$, since the difference is bounded by $2\max(A)+Lk\leq 3\max(A)\leq \epsilon N$. 
Since $N>3\max(A)/\epsilon> N_0$, we may apply Theorem \ref{KLRadziwill} to get the third term of the above inequality, which is an over-estimate of the sum
$$ \Bigr\lvert\displaystyle\sum_{n\in M''\cap [z+t_sLk, z+t_rLk)}\mu(m)\Bigr\rvert $$
except over the intervals $[z+t_iLk, z+(t_i+1)Lk)$ and $[z+u_iLk, z+(u_i+1)Lk)$ for $s\leq i<r$. Finally, the total error on these intervals is bounded by $(r-s)2Lk\leq [N/\max(A)]2Lk\leq \epsilon N$, which gives the second error term of $\epsilon N$.
\end{proof}

\begin{cor}\label{odometer}
Let $M\subseteq \mathbb{Z}$ be an accc.
Then $M$ is orthogonal to the M\"{o}bius function.
\end{cor}

\begin{proof}
This is a direct consequence of Theorem \ref{multiplication} with $l=1$ and $n_1=0$.
\end{proof}

\begin{thm}\label{acccsubshift}
Let $(X,T)$ be an accc rank-one subshift. Then $(X, T)$ is M\"{o}bius disjoint.
\end{thm}

\begin{proof}
Let $$F=\{f_{n_1}f_{n_2}\cdots f_{n_l}: l\in\mathbb{N} \text{ and } n_1,n_2,\dots, n_l\in \mathbb{Z}\} \cup\{f_{\text{const}}\}$$ where $f_n:X\rightarrow \mathbb{R}$ is the projection onto the $n$-th coordinate and $f_{\text{const}}:X\rightarrow \mathbb{R}$ is the constant function $f_{\text{const}}(x)=1$. Note that $F$ separates points since for every $x,y\in X$ with $x\neq y$ there exists $n\in \mathbb{Z}$ such that $f_n(x)=x(n)\neq y(n)=f_n(y)$. Since $F$ contains a non-zero constant function, by the Stone--Weierstrass Theorem the algebra generated by functions in $F$ is dense in the space of all continuous functions on $X$ with the uniform convergence topology. Furthermore, since Sarnak's property is closed under taking the limit with uniform convergence topology, it is enough to show Sarnak's property for every continuous function in the algebra (closed under taking linear combinations and multiplication) generated by $F$. Note that since $F$ is closed under multiplication, the algebra generated by $F$ is equal to
\[
\{c_1g_1+\cdots+c_lg_l : l\in \mathbb{N}, c_1,c_2,\dots,c_l\in\mathbb{R}, \mbox{ and } g_1,g_2,\dots,g_l\in F\}.
\]

We next show Sarnak's property for $f=f_{n_1}f_{n_2}\cdots f_{n_l}\in F$. Let $x\in X$. Then  $M_x=\{n\in\mathbb{Z}: x(n)=0\}$ is an accc. For any subset $I\subseteq \{1, \dots, l\}$, say $I=\{i_1,\dots, i_p\}$, and for $N\in \mathbb{N}$, we have
\[
\frac{1}{N}\sum_{\substack{n+n_{i_1},\dots,n+n_{i_p}\in M_x\\ 1\leq n\leq N}} \mu(n)
=\frac{1}{N}\sum_{1\leq n\leq N} \mu(n) (1-x(n+n_{i_1}))\cdots (1-x(n+n_{i_p})),
\]
which approaches 0 as $N\to\infty$ by Theorem \ref{multiplication}. Now observe
\begin{align*}
&\frac{1}{N}\sum_{1\leq n\leq N} \mu(n) f_{n_1}(T^nx)f_{n_2}(T^nx)\cdots f_{n_l}(T^nx)\\ =&\frac{1}{N}\sum_{1\leq n\leq N} \mu(n) x(n+n_1)x(n+n_2)\cdots x(n+n_l)\\=&\frac{1}{N}\sum_{1\leq n\leq N} \mu(n) (1-(1-x(n+n_1)))\cdots (1-(1-x(n+n_l))) \\
=&\sum_{\substack{I\subseteq\{1,\dots, l\} \\ I=\{i_1,\dots, i_p\} }}(-1)^p\frac{1}{N}\sum_{1\leq n\leq N} \mu(n)(1-x(n+n_{i_1}))\cdots (1-x(n+n_{i_p})).
\end{align*}
Thus, by applying Theorem \ref{multiplication} as above $2^l$-many times, we get $$\frac{1}{N}\sum_{1\leq n\leq N} \mu(n) f_{n_1}(T^nx)f_{n_2}(T^nx)\cdots f_{n_l}(T^nx)\rightarrow 0.$$

Finally, we show Sarnak's property for $f=c_1g_1+\cdots+c_lg_l$ assuming that each $g_i$ satisfies Sarnak's property. Here, for every $\epsilon>0$ there exists $N_0\in \mathbb{N}$ such that for every $N\geq N_0$ we have
\[
\Bigr\lvert\frac{1}{N}\sum_{n\leq N} \mu(n) f(T^nx)\Bigr\rvert= \Bigr\lvert\sum_{i=1}^l c_i\frac{1}{N}\sum_{n\leq N} \mu(n) g_i(T^nx)\Bigr\rvert \leq \left(\sum_{i=1}^l \lvert c_i\rvert\right) \epsilon.
\]
\end{proof}

\begin{cor}
Let $(X,\mu,T)$ be a symbolic rank-one transformation that is measure theoretically isomorphic to an odometer. Then the rank-one subshift $(X, T)$ is M\"{o}bius disjoint.
\end{cor}
\begin{proof}
This is a direct consequence of Theorem \ref{Getal} Clause (b) and Theorem \ref{acccsubshift}.
\end{proof}

\section{Generalized Katok's Maps}
In this section we verify Sarnak's conjecture for a class of rank-one subshifts
which generalize Katok's map studied in \cite{ALR}. We first define this class. Recall that a generating sequence $\{v_n\}_{n\in\mathbb{N}}$ of a rank-one subshift is defined recursively from the cutting parameter $\{r_n\}_{n\in\mathbb{N}}$ and the spacer parameter $\{s_{n, i}\}_{n\in\mathbb{N},0< i\leq r_n}$ by $v_0=0$ and $$v_{n+1}=v_n1^{s_{n,1}}v_n1^{s_{n,2}}\cdots v_n1^{s_{n, r_n}}$$ for $n\in \mathbb{N}$. For each integer $m\geq 2$, let $\mathcal{K}_m$ be the set of all infinite rank-one words $V\in 2^\mathbb{N}$ with generating sequences $\{v_n\}_{n\in\mathbb{N}}$ such that there are natural numbers $0\leq t_{n,1},t_{n,2},\dots,t_{n,m}\leq m-1$ for each $n\in\mathbb{N}$, satisfying \vskip 3pt
\begin{enumerate}
\item $r_n$ is divisible by $m$. \vskip 3pt
\item $s_{n,i}=\displaystyle{t_{n,\lceil\frac{m}{r_n}i\rceil}}$ for $0<i\leq r_n$. \vskip 3pt
\item $\displaystyle\liminf_{n\rightarrow \infty}\frac{\log\log(r_n)}{\log\log\log(|v_n|)}\geq 2$.\vskip 6pt
\end{enumerate}
Note that the original Katok's map is a special case in $\mathcal{K}_2$, and Condition (3) is weaker than the condition in \cite{ALR} which requires $\lim_{n\rightarrow \infty}r_n/|v_n|=+\infty$. Let $\mathcal{K}=\bigcup_{m\geq 2}\mathcal{K}_m$. We show Sarnak's conjecture for $(X_V,T)$ for all $V\in \mathcal{K}$.

\begin{thm}\label{GKatok} Let $V\in \mathcal{K}$, $x\in X_V$, and $n_1,n_2,\dots,n_l$ be integers.
Let $M_x=\{n\in \mathbb{Z} :  x(n)=0\}$ and
\[
M'_x=\{n\in \mathbb{Z} \ : \ n+n_1,n+n_2,\dots,n+n_l\in M_x\}.
\]
Then $M_x'$ is orthogonal to the M\"{o}bius function.
\end{thm}

\begin{proof}
Fix $m\geq 2$ such that $V\in\mathcal{K}_m$. We may assume $M_x$ is nonempty since otherwise $M'_x$ is the empty set and therefore orthogonal to the M\"{o}bius function. By Condition (3) of the definition of $\mathcal{K}_m$, there exists $N_0\in \mathbb{N}$ such that for every $n\geq N_0$ we have $\log\log(r_n)\geq 2\log\log\log(|v_n|)$. For each $n\in\mathbb{N}$, let $A_n=\{0\leq i<|v_n|: v_n(i)=0\}$. Then $0\in A_n$, $A_n$ is a building block of $M_x$ and we have 
\[
\max(A_n)= |v_n|-t_{n-1,m}-t_{n-2,m}-\cdots-t_{0,m}-1=|v_n|-O_m(1).
\]
For every $n\in\mathbb{N}$ fix a non-decreasing sequence $\{a_{j,A_n}\}_{j\in\mathbb{Z}}$ of integers such that ${M_x=\bigcup_{j\in\mathbb{Z}} (A_{n}+a_{j,A_n})}$ and that $a_{j+1,A_n}-a_{j,A_n}>\max(A_n)$ or $a_{j+1,A_n}=a_{j,A_n}$, for every $j\in \mathbb{Z}$.

Let $\delta=\max(\lvert n_1\rvert,\dots,\lvert n_l\rvert)$. For a moment, fix $j\in\mathbb{Z}$ and $n\in\mathbb{N}$, and consider the set $A_{n+1}+a_{j,A_{n+1}}$, which is one of the translations of the building block $A_{n+1}$ in $M_x$. By the definition of $\mathcal{K}_m$, we can write
$$ v_{n+1}=(v_n1^{t_{n,1}})^{\frac{r_n}{m}}(v_n1^{t_{n,2}})^{\frac{r_n}{m}}\cdots (v_n1^{t_{n,m}})^{\frac{r_n}{m}}=u_{n,1}u_{n,2}\cdots u_{n,m} $$
where $u_{n,k}=(v_n1^{t_{n,k}})^{\frac{r_n}{m}}$ for $1\leq k\leq m$. For each $0\leq \ell< m$, let $\sigma_{n,\ell}$ be the starting position of $u_{n,\ell+1}$ in $v_{n+1}$, that is,
$$ \sigma_{n,\ell}=(|v_n|+t_{n,1}+\cdots+|v_n|+t_{n,\ell})\frac{r_n}{m}, $$
and let $\sigma_{n,m}=|v_{n+1}|$. We say $B\subseteq \mathbb{Z}$ is a {\it complete congruency class mod $q$} (ccc mod $q$) on an interval $[a,b)$ if for every $n\in [a,b-q)$, $n\in B$ iff $n+p\in B$. Then for each $0\leq \ell<m$, $A_{n+1}$ is a ccc mod $|v_n|+t_{n,\ell+1}$ on $[\sigma_{n,\ell}, \ \sigma_{n,\ell+1})$, and $A_{n+1}+a_{j, A_{n+1}}$ is a ccc mod $|v_n|+t_{n,\ell+1}$ on $[a_{j, A_{n+1}}+\sigma_{n,\ell}, \ a_{j, A_{n+1}}+\sigma_{n,\ell+1})$. Since $M_x\cap [a_{j, A_{n+1}},\ a_{j+1, A_{n+1}})=A_{n+1}+a_{j, A_{n+1}}$, it follows that $M_x$ is a ccc mod $|v_n|+t_{n,\ell+1}$ on $[a_{j, A_{n+1}}+\sigma_{n,\ell},\ a_{j, A_{n+1}}+\sigma_{n,\ell+1})$, and that $M'_x$ is a ccc mod $|v_n|+t_{n,\ell+1}$ on $[a_{j, A_{n+1}}+\sigma_{n,\ell}+\delta, \ a_{j, A_{n+1}}+\sigma_{n,\ell+1}-\delta)$.

For any $n\in\mathbb{N}$ and $0\leq \ell<m$, let
$$ C_{n,\ell}=\displaystyle\bigcup_{j\in\mathbb{Z}}\left(M'_x\cap[a_{j,A_{n+1}}+\sigma_{n,\ell}+\delta, \ a_{j,A_{n+1}}+\sigma_{n, \ell+1}-\delta)\right), $$
$q_{n,\ell}=|v_n|+t_{n,\ell+1}$, and $\lambda_{n,\ell}=q_{n,\ell}\frac{r_n}{m}-2\delta$. Then each $C_{n,\ell}$ is a union of arithmetic progressions each of which has approximately $r_n/m$-many terms of common difference $q_{n,\ell}$, contained within an interval of length no longer than $\lambda_{n,\ell}$.
Let $C_n=\bigcup_{\ell=0}^{m-1}C_{n,\ell}$. We will find large enough $n$ to apply Theorem \ref{KLRadziwill} to each $C_{n,\ell}$.

Fix $0< \epsilon< \frac{1}{100}$ and a bound $L_0$ corresponding to $\epsilon$ given by Theorem \ref{KLRadziwill}. For each $n\geq N_0$, we first find $L_n\geq L_0$ such that for all $0\leq \ell<m$,
\[
\sum_{\substack{p\vert q_{n,\ell}\\ p \text{ prime}}}\frac{1}{p}\leq\  (1-\epsilon)\!\!\sum_{\substack{p\leq L_n\\ p \text{ prime}}}\frac{1}{p}.
\]
By \cite{KLR} Lemma 3.2 we have
\[
\sum_{\substack{p\vert q\\ p \text{ prime}}}\frac{1}{p}\leq \log\log\log(q)+O(1)
 \]
and from \cite{Apo} Theorem 4.12 we have
\[ \sum_{\substack{p\leq L\\ p \text{ prime}}}\frac{1}{p}=\log\log(L)+O(1).
\]
Since $\log\log\log(q_{n,\ell})=\log\log\log(|v_n|)+O(1)$ and $\log\log(r_n)\geq 2\log\log\log(|v_n|)$, there exists $N_1\geq N_0$ such that, for every $n\geq N_1$, we have $|v_n|\geq m$ and, setting $L_n=\epsilon r_n/4m$, $L_n\geq L_0$. Now $\log\log(L_n)=\log\log(r_n)+O_\epsilon(1)$ and
\[
\sum_{\substack{p\vert q_{n,\ell}\\ p \text{ prime}}}\frac{1}{p}\leq \ (1-\epsilon)\!\!\sum_{\substack{p\leq L_n\\ p \text{ prime}}}\frac{1}{p}
\]
for all $0\leq \ell<m$. Note that $L_n\geq L_0$ but
$$L_nq_{n,\ell}\leq\epsilon \displaystyle\frac{r_n}{4m} (|v_n|+m)\leq \epsilon \frac{r_n|v_n|}{2m}\leq \epsilon \frac{|v_{n+1}|}{2m}$$
for all $0\leq \ell<m$.

Finally let $L=L_{N_1}$ and let
$$ N_2=\max\{N_1, N_0(q_{N_1,\ell}, L): 0\leq \ell<m\}$$
where $N_0(\cdot, \cdot)$ is given by Theorem \ref{KLRadziwill}. Denote $H=|v_{N_1+1}|$ and let $N\geq 2N_2H/\epsilon>N_2$. By Theorem \ref{KLRadziwill}, for each $p=0,\dots, m-1$ and $q=|v_{N_1}|+p$ we have
\[
\displaystyle\sum_{j=0}^{N/Lq}\sum_{a=0}^{q-1}\Bigr\lvert \sum_{\substack{i\in[z_q+jLq,z_q+(j+1)Lq)\\i\equiv a \bmod{q}}}\mu(i) \Bigr\rvert\leq\epsilon N.
\]
On the other hand, let $s=\min\{j\in\mathbb{Z}: 1\leq a_{j, A_{N_1+1}}\leq N\}$, $r=\max\{j\in\Z: 1\leq a_{j, A_{N_1+1}}\leq N\}$, $E_\ell=C_{N_1, \ell}\cap [a_{s, A_{N_1+1}},\ a_{r, A_{N_1+1}})$, and $E=\bigcup_{\ell=0}^{m-1}E_{\ell}$. Then the sum
$$ \displaystyle\sum_{q=|v_{N_1}|}^{|v_{N_1}|+m-1}\ \sum_{j=0}^{N/Lq}\sum_{a=0}^{q-1}\Bigr\lvert \sum_{\substack{i\in[z_q+jLq,z_q+(j+1)Lq)\\i\equiv a \bmod{q}}}\mu(i) \Bigr\rvert
$$
is an over-estimate of
$$ \Bigr\lvert\displaystyle\sum_{\substack{i\in E \\ 1\leq i\leq N}}\mu(i)\Bigr\rvert $$
with an error no bigger than
\[ m(r-s)2L(|v_{N_1}|+m)\leq m\displaystyle\frac{N}{H}2\epsilon\frac{H}{2m}=\epsilon N.
\]
To justify this error estimate, note that each application of Theorem \ref{KLRadziwill} with $q=|v_{N_1}|+p$ where $p=q_{N_1,\ell}$ gives an over-estimate of
$$ \Bigr\lvert \displaystyle\sum_{\substack{i\in E_p \\ 1\leq i\leq N}}\mu(i)\Bigr\rvert $$
with an error occurring near each end of the interval
$$[a_{j,A_{N_1+1}}+\sigma_{N_1,\ell}+\delta,\ a_{j, A_{N_1+1}}+\sigma_{N_1, \ell+1}-\delta) $$
within an interval of length $Lq_{N_1,\ell}$.

Finally, since $|(M'_x\cap [1, N])\triangle E|\leq 2H\leq \epsilon N$, we obtain
\[
\Bigr\lvert\sum_{\substack{i\in M'_x\\ 1\leq i\leq N}} \mu(i)\Bigr\rvert \leq \epsilon N +\Bigr\lvert\displaystyle\sum_{\substack{i\in E \\ 1\leq i\leq N}}\mu(i)\Bigr\rvert \leq \epsilon N+m\epsilon N+\epsilon N=(m+2)\epsilon N.
\]
\end{proof}

With an argument similar to the proof of Theorem \ref{acccsubshift}, we obtain the following corollary of Theorem \ref{GKatok}.
\begin{cor}
Let $V\in\mathcal{K}$. Then $(X_V, T)$ is M\"{o}bius disjoint.
\end{cor}

\section{Further Generalizations}
In this last short section we note that the results in Section 3 can be generalized further. We use the same notation from previous sections for rank-one subshifts. In particular, let $A_n=\{0\leq i\leq |v_n|-1:v_n(i)=0\}$. In general, our techniques can only be applied in case there exists $m\in\mathbb{N}$ such that $A_n$ for arbitrarily large $n$ can be approximated with a union of long arithmetic progressions with at most $m$-many common differences. In order to apply Theorem \ref{KLRadziwill}, these common differences and the lengths of the arithmetic progressions need to be constrained by (\ref{longap}), which usually results in some growth conditions on the cutting parameter and moderation (or bounded) conditions on the spacer parameter of the rank-one subshift.

Here we specify one concrete class of rank-one subshifts that is broader than $\mathcal{K}$ and satisfies Sarnak's conjecture. Define
\[
C_n=\{1, r_n\}\cup\{2\leq i\leq r_n-1 : s_{n,i-1}\neq s_{n,i}\}
\]
and enumerate the members of $C_n$ in increasing order as $c_{n,1},c_{n,2},\dots,c_{n,p_n}$. The arguments in Section 3 can be repeated to show Sarnak's conjecture for $(X_V,T)$ under the following conditions:
\begin{enumerate}
\item[(i)] $\displaystyle\lim_{n\rightarrow \infty}\frac{\log\log(r_n/p_n)}{\log\log(r_n)}=1$, \vskip 3pt
\item[(ii)] $\displaystyle K=\limsup_{n\rightarrow\infty}\frac{\sum_{i=1}^{r_n}s_{n,i}}{r_n|v_n|}<+\infty$, \vskip 3pt
\item [(iii)] $\displaystyle\limsup_{n\rightarrow \infty}\frac{\log\log(r_n)}{\log\log\log(|v_n|)}> 1$, and \vskip 6pt
\item [(iv)] there exists $m\in\mathbb{N}$ such that for every $\epsilon>0$ there exists $N\in\mathbb{N}$ such that for $n\geq N$ there exists $A\subseteq [1,r_n]$ with $\lvert A\rvert\geq (1-\epsilon)r_n$ and $\displaystyle\lvert\{ s_{n,a}:a\in A\}\rvert\leq m$.
\end{enumerate}

We sketch the proof to illustrate how the conditions are applied. Fix $0<\epsilon<\frac{1}{100}$. The strategy is to approximate $A_{n+1}$ (with a suitably defined, large enough $n$) with a union of arithmetic progressions with common difference $q_i=|v_n|+s_{n,c_{n,i}}$ and length $L_i=c_{n,i+1}-c_{n,i}$ such that (\ref{longap}) is satisfied. Note that condition (i) let us guarantee that $L_i\geq \sqrt{\frac{r_n}{p_n}}$ by removing at most $\epsilon |v_{n+1}|$ many points, that is, we may assume that $\log\log(r_n)-\log\log(L_i)=O(1)$. Condition (ii) let us guarantee $q_i\leq |v_{n}|^2$ by removing at most $\epsilon |v_{n+1}|$ many points, that is, we may assume $\log\log\log(|v_{n}|)-\log\log\log(q_i)=O(1)$. In light of conditions (i) and (ii), condition (iii) let us choose $n$ large enough such that (\ref{longap}) holds for every $i$ with $q=q_i$ and $L=L_i$. Finally, condition (iv) guarantees that there are at most $m$-many different values of $q_i$ (and therefore we only need to apply Theorem \ref{KLRadziwill} $m$-many times) by removing at most $\epsilon |v_{n+1}|$ many points.

Conditions (i)--(iv) define a class of rank-one subshifts that is more general than $\mathcal{K}$. In addition, they also include rank-one subshifts correspondent to certain flat stacks. A flat stack is a rank-one transformation $T$ on a probability measure space $(X,\mu)$ (see \cite{F}, Definition 2) with the extra condition that for every $\epsilon>0$ we can choose $F$ in Definition 2 in \cite{F} such that $\mu(T^hF\Delta F)\leq\epsilon \mu(F)$. In particular, if $(X_V,T)$ is a rank-one subbshift correspondent to a flat stack then
\[
\lim_{n\rightarrow \infty}\frac{\lvert\{1\leq i\leq r_n:s_{n,i}\neq 0\}\rvert}{r_n}=0.
\]
Thus,
it satisfies conditions (ii) and (iv).

\end{document}